\documentclass{article}
\usepackage{amssymb,amsmath,amsthm,verbatim}
\usepackage{ifthen}

\newtheorem{theorem}{Theorem}
\newtheorem{proposition}[theorem]{Proposition}
\newtheorem{lemma}[theorem]{Lemma}

\newtheorem{claim}[theorem]{Claim}
\newtheorem{corollary}[theorem]{Corollary}
\newtheorem{conjecture}{Conjecture}

\theoremstyle{definition}
\newtheorem{definition}{Definition}
\newtheorem{notation}[definition]{Notation}

\theoremstyle{remark}
\newtheorem{remark}{Remark}

\newcommand{\mythmname}{}
\newtheoremstyle{mytheorem}
	{3pt}
	{3pt}
	{\it}
	{}
	{}
	{{\bf .}}
	{.5em}
	{\mythmname{\ifthenelse{ \equal{#3}{} }{}{\ (\thmnote{#3})}}}
\theoremstyle{mytheorem}

\newcommand{\tensor}{\otimes}

\newcommand{\del}[2]{\ensuremath{\frac{\partial #1}{\partial #2}}}

\newcommand{\delz}{\del{}{z}}

\newcommand{\delu}{\del{}{u}}

\newcommand{\GA}{{\rm GA}}
\newcommand{\MA}{{\rm MA}}
\newcommand{\EA}{{\rm EA}}
\newcommand{\TA}{{\rm TA}}

\newcommand{\GL}{{\rm GL}}

\newcommand{\Venereau}{V\'en\'ereau }
\newcommand{\Vtype}{V\'en\'ereau-type }

\DeclareMathOperator{\Aut}{Aut}
\DeclareMathOperator{\End}{End}

\DeclareMathOperator{\Spec}{Spec}

\newcommand{\IC}{\mathbb{C}}

\newcommand{\IN}{\mathbb{N}}
\newcommand{\IQ}{\mathbb{Q}}

\newcommand{\IZ}{\mathbb{Z}}

\usepackage{diagrams}

\begin{document}

\title{\Vtype polynomials as potential counterexamples}

\author{Drew Lewis\thanks{ {\em Current affiliation}: University of Alabama, amlewis@as.ua.edu}\\ Washington University in St. Louis \\ andrew@math.wustl.edu}

\maketitle
\begin{abstract}
We study some properties of the \Venereau polynomials $b_m=y+x^m(xz+y(yu+z^2)) \in \IC[x,y,z,u]$, a sequence of proposed counterexamples to the Abhyankar-Sathaye embedding conjecture and the Dolgachev-Weisfeiler conjecture.  It is well known that these are hyperplanes and residual coordinates, and for $m \geq 3$, they are $\IC[x]$-coordinates.  For $m=1,2$, it is only known that they are 1-stable $\IC[x]$-coordinates.  We show that $b_2$ is in fact a $\IC[x]$-coordinate.  We introduce the notion of \Vtype polynomials, and show that these are all hyperplanes and residual coordinates.  We show that some of these \Vtype polynomials are in fact $\IC[x]$-coordinates; the rest remain potential counterexamples to the aforementioned conjectures.  For those that we show to be coordinates, we also show that any automorphism with one of them as a component is stably tame.  The remainder are stably tame, 1-stable $\IC[x]$-coordinates.
\end{abstract}



\section{Introduction}

Let $A$ (and all other rings) be a commutative ring with one throughout.  An {\em $A$-coordinate} (if $A$ is understood, we simply say {\em coordinate}; some authors prefer the term {\em variable}) is a polynomial $f \in A^{[n]}$ for which there exist $f_2,\ldots f_n \in A^{[n]}$ such that $A[f,f_2,\ldots,f_n]=A^{[n]}$.  It is natural to ask when a polynomial is a coordinate; this question is extremely deep and has been studied for some time.  There are several conjectures concerning the identification of coordinates.  First, the well-known Embedding Conjecture asserts, in short, that hyperplanes are coordinates:

\begin{conjecture}\emph{(Abhyankar-Sathaye)} \label{AS}
Let $A$ be a $\IQ$-algebra and $f \in A^{[n]}$.  If $A^{[n]}/(f) \cong A^{[n-1]}$, then $f$ is a coordinate of $A^{[n]}$.
\end{conjecture}

This is known only for $n=2$ with $A$ a polynomial ring over a field (\cite{AM}, \cite{Suzuki}, \cite{RusSathaye}).  We remark that the condition of $A$ being a $\IQ$-algebra is necessary, even with $n=2$.

If we restrict our attention to $A=\IC^{[r]}$, a special case of the Dolgachev-Weisfeiler Conjecture asserts that residual coordinates are coordinates:

\begin{conjecture}\label{DW}
Suppose $A=\IC^{[r]}$, and let $f \in A^{[n]}$.  If $A[f] \hookrightarrow A^{[n]}$ is an affine fibration, then $f$ is an $A$-coordinate.
\end{conjecture}

In studying these kinds of problems, \Venereau \cite{Vthesis} and Berson \cite{BThesis} independently wrote down a polynomial in four variables, now known as the \Venereau polynomial, that satisfies the hypotheses of these two conjectures (with $A=\IC[x]$).  It is often given as the first polynomial in the sequence
\begin{equation}\label{Vpoly}
b_m := y+x^m(xz+y(yu+z^2))
\end{equation}
\Venereau was able to show that $b_m$ is a coordinate for $m \geq 3$, but the $m=1$ and $m=2$ cases proved elusive.  Subsequently, many coordinate-like properties were shown about $b_1$ and $b_2$ ( e.g. \cite{EMV}, \cite{Freud}, \cite{KZPoly}).

We show that the second \Venereau polynomial $b_2$ is a coordinate.  We go on to highlight a key difference between $b_1$ and $b_2$.  In the process, we describe a more general class of polynomials we term {\em \Vtype polynomials}.  These are all hyperplanes and strongly residual coordinates; we show many of them to be coordinates, but are left with a class of potential counterexamples (including $b_1$) to the above conjectures.  We go on to show that every coordinate-like property known to us to hold for $b_1$ also holds for any \Vtype polynomial.

\section{Preliminaries}
For a ring $A$, we will adopt the standard notation of using $\GA_n(A)$ to denote the general automorphism group $\Aut _{\Spec A} \Spec A^{[n]}$, and $\MA_n(A)$ for the endomorphism group $\End _{\Spec A} \Spec A^{[n]}$.  Note that $\GA_n(A)$ is naturally anti-isomorphic to $\Aut _A A^{[n]}$ (in fact, some authors choose to define it as such).  $\EA_n(A)$ denotes the elementary group generated by automorphisms fixing $n-1$ variables, and $\TA_n(A):=\left<\EA_n(A),\GL_n(A)\right>$ is called the tame subgroup.  Given $\phi \in \GA_n(A)$, we will use $J\phi$ for the Jacobian determinant.  We also require the following definitions:
\begin{definition}Let $A$ be a ring and $f \in A^{[n]}$.
\begin{enumerate}
\item $f$ is a {\em hyperplane} if $A^{[n]}/(f) \cong A^{[n-1]}$.
\item $f$ is a {\em residual coordinate} if $A[f] \hookrightarrow A^{[n]}$ is an affine fibration; that is, if $A[f] \hookrightarrow A^{[n]}$ is flat and $A^{[n]} \tensor _{A[f]} \kappa (\mathfrak{p}) \cong \kappa (\mathfrak{p}) ^{[n-1]}$ for all $\mathfrak{p} \in \Spec A[f]$.
\item Let $x \in A \setminus A^*$ be a non-zero divisor.  $f$ is a {\em strongly $x$-residual coordinate} if $\bar{f}$ is a coordinate in $(A/xA)^{[n]}$ and $f$ is a coordinate in the localization $(A_x) ^{[n]}$.  If $x$ is clear from context, we may simply say {\em strongly residual coordinate}.
\item $f$ is an {\em $r$-stable coordinate} (or just {\em stable coordinate}) if $f$ is a coordinate in $A^{[n+r]}$ for some $r \in \IN$.
\end{enumerate}
\end{definition}

We are primarily interested in the case $A=\IC^{[r]}$, in which case one easily checks that strongly residual coordinates are indeed residual coordinates.  As observed in \cite{FreudDaigle}, a result of Asanuma \cite{Asanuma} combined with the Quillen-Suslin Theorem yields

\begin{theorem}\label{Asanuma}Let $A=\IC^{[r]}$ for some $r \in \IN$.  Then every residual coordinate (and hence, every strongly residual coordinate) is a stable coordinate.
\end{theorem}

A recent result in this same sprit that we will find useful is due to Berson, van den Essen, and Wright (a special case of Theorem 4.5 from \cite{BEW})
\begin{theorem}\label{BEWtheorem}
Let $\phi \in \GA_n(\IC[x])$ with $J\phi=1$.  If $\phi \in \TA_n(\IC[x,x^{-1}]) $, and $\bar{\phi} \in \EA_n(\IC)$ (where $\bar{\phi}$ denotes the image modulo $x$), then $\phi$ is stably tame.
\end{theorem}

We also make use of the following well known result, a special case of the overring principle (\cite{Arnobook} Proposition 1.1.7, for example):
\begin{theorem}\label{overring}
Let $A$ be a reduced ring and $x \in A$ a non-zero divisor.  Let $\phi \in \MA_n(A) \subset \MA_n(A_x)$ with $J\phi \in A^*$.  Then $\phi \in \GA_n(A)$ if and only if $\phi \in \GA_n(A_x)$
\end{theorem}
One can produce many exotic objects from this kind of construction.  The famous Nagata map, perhaps the simplest example of a wild (i.e.\ not tame) automorphism, arises in this manner as an element of $\GA_2(\IC[x]) \subset \GA_3(\IC)$:
\begin{align}
\sigma &= \left(x,y+x(xz-y^2),z+2y(xz-y^2)+x(xz-y^2)^2\right) \label{Nagata} \\
&=\left(x,y,z+\frac{y^2}{x}\right) \circ \left(x,y+x^2z,z\right) \circ \left(x,y,z-\frac{y^2}{x}\right) \notag
\end{align}

We use this approach to construct a variety of strongly residual coordinates.  Consider $\alpha \in \GA_n(\IC[x,x^{-1}][y])$, and let $Q \in \IC[x]^{[n]}=\IC[x][z_1,\ldots,z_n]$ such that $\alpha(Q) \in \IC[x,y]^{[n]}$.  Then $f:=y+x\alpha(Q)$ is a strongly $x$-residual coordinate; indeed, $\bar{f} \equiv y \pmod{x}$ is trivially a coordinate; and over $\IC[x,x^{-1}]$, we have $\phi(y)=f$ where $\phi \in \GA_{n+1}(\IC[x,x^{-1}])$ is given by
\begin{equation} \label{construction}
\phi = (y+xQ,z_1,\ldots,z_n) \circ \alpha
\end{equation}
Conjecture \ref{DW} asserts that these should all be coordinates.  If $Q \in x^m \IC[x]^{[n]}$ for a sufficiently large $m$, then we quickly see that $f$ is in fact a coordinate.

\begin{theorem}\label{bigm}
Let $\alpha =(f_1,\ldots,f_n) \in \GA_{n}(\IC[x,x^{-1}][y])$ and $Q \in \IC[x]^{[n]}$.  Then for all $m>>0$, $y+x^mQ(f_1,\ldots,f_n)$ is a $\IC[x]$-coordinate.  To be precise, write $\alpha^{-1}=(g_1,\ldots, g_n)$ and define, for $k>0$ and $1 \leq j \leq n$, 
\begin{align*}
q_{j,k} &= \min \{q\ |\ x^{qk} \del{ ^k g_j}{y^k} (y,f_1,\ldots,f_n) \in \IC[x,y]^{[n]} \}  
\end{align*}
Also define
\begin{align}
m_1 &:= \max _
{1 \leq j \leq n;\ 0<k} 
\{q_{j,k}\} \label{m1} \\ 
m_2 &:=\min \{m \in \IN\ |\ x^mQ(f_1,\ldots,f_n) \in \IC[x,y]^{[n]}\} \label{m2}
\end{align}
Then $y+x^mQ(f_1,\ldots,f_n)$ is a coordinate if $m \geq m_1+m_2$.
\end{theorem}
\begin{proof}
This is just an application of Taylor's formula.  Let $$\phi=\alpha ^{-1} \circ (y+x^mQ,z_1,\ldots,z_n) \circ \alpha$$ and compute 
\begin{align*}
\phi(z_j)&=g_j(y+x^mQ(f_1,\ldots,f_m),f_1,\ldots,f_n) \\
&=\sum _{k=0} \frac{1}{k!} \del{ ^k g_j}{ y^k}(y,f_1,\ldots,f_n)(x^mQ(f_1,\ldots,f_n))^k \\
&=z_j + \sum _{k=1} \frac{1}{k!} x^{k(m-m_2-m_1)} \left(x^{km_1}\del{^k g_j}{ y^k}(y,f_1,\ldots,f_n)\right)(x^{m_2}Q(f_1,\ldots,f_n))^k 
\end{align*}
It is now immediate from \eqref{m1} and \eqref{m2} that $\phi(z_j) \in \IC[x,y]^{[n]}$.  One quickly checks that $J\phi = 1$, so by Theorem \ref{overring}, $\phi \in \GA_n(\IC[x])$.
\end{proof}

\section{The \Venereau polynomials}
We start by defining a derivation on $\IC[y,z,u]$ by 
\begin{equation}\label{D}
D:=y\delz-2z\delu
\end{equation}
$D$ is triangular and thus locally nilpotent.  We also define 
\begin{align}\label{pvw}
p &:= yu+z^2 & v&:=xz+yp & w&:=x^2u-2xzp-yp^2
\end{align}
Note that $\ker D = \IC[y,p]$, and  $\exp(pD)\in \GA_3(\IC)$ is (up to a change of variables) the Nagata map \eqref{Nagata}.  By defining $Dx=0$, $D$ naturally extends to a derivation on $\IC[x]^{[3]}$ and thus $\IC[x,x^{-1}]^{[3]}$ with kernel $\IC[x,x^{-1}][y,p]$.  Thus we may consider $\psi := \exp(\frac{p}{x}D) \in \GA_3(\IC[x,x^{-1}])$.  One quickly computes
\begin{equation}\label{psi}
\psi = \left(y, \frac{v}{x}, \frac{w}{x^2}\right)
\end{equation}

We note that since $\psi$ fixes $p$, we quickly obtain the relation

\begin{equation}\label{x2p}
yw+v^2=x^2p
\end{equation}

As in \eqref{construction}, we thus see the \Venereau polynomials naturally as the $y$-component of
\begin{equation}
(y+x^{m+1}z,z,u) \circ \psi
\end{equation}
Applying Theorem \ref{bigm} in this case recovers V\'en\'ereau's result that $b_m$ is a coordinate for $m \geq 3$.  A slight variation of that technique shows that $b_2$ is a coordinate.
\begin{theorem}\label{f2}The second \Venereau polynomial $b_2=y+x^2(xz+y(yu+z^2))$ is a coordinate.
\end{theorem}
\begin{proof}[Abbreviated Proof]
Define $\varphi _m \in \GA_3(\IC[x,x^{-1}])$ by
\begin{equation}\label{varphi}
\varphi _m = \psi ^{-1} \circ (y,z-\frac{1}{2}x^{m+1}u,u) \circ (y+x^{m+1}z,z,u) \circ \psi
\end{equation}
One quickly checks that $\varphi _m (y)=b_m$ and $J\varphi _m=1$; so by Theorem \ref{overring}, all that is left to check is that for $m \geq 2$, $\varphi _m \in \MA_3(\IC[x])$.  This is straightforward to check and is shown in more generality below in Theorem \ref{coordinate}.
\end{proof}

\begin{remark}We can also see $\varphi _m$ as an exponential: namely, let $d$ be the jacobian derivation given by $d=J(p+\frac{1}{2}x^{m-2}vw,w,\cdot)$.  One can check that $d$ is locally nilpotent and $\varphi _m = \exp(\frac{1}{2}x^{m-1}d)$.
\end{remark}

Prior to our promised discussion of \Vtype polynomials, we would like to highlight a significant difference between $b_1$ and $b_2$.  For this, we define:
\begin{notation}
Let $\alpha \in \GA_n(\IC[x,x^{-1}])$.  Define $$\alpha ^*(\GA_n(\IC[x]))=\{\alpha^{-1} \circ \phi \circ \alpha\ |\ \phi \in \GA_n(\IC[x])\} \subset \GA_n(\IC[x,x^{-1}])$$
\end{notation}

The proof of Theorem \ref{f2} shows that $b_2$ is a coordinate of $\varphi _2 \in \psi ^*(\GA_3(\IC[x]))$.  More strongly, defining $\tilde{\psi} \in \GA_3(\IC[x,x^{-1}])$ by
\begin{equation}\label{psitilde}
\tilde{\psi} := (y,v,w) = (y,xz,x^2u) \circ \psi 
\end{equation}
we may even write
\begin{equation*}
\varphi _m = \tilde{\psi} \circ (y,z-\frac{1}{2}x^mu,u) \circ (y+x^mz,z,u) \circ \tilde{\psi}
\end{equation*}
and thus see $\varphi _2 \in \tilde{\psi} ^* (\GA_3(\IC[x]))$.  However, this is too much to ask of $b_1$.

\begin{theorem}\label{b1fail}
There is no automorphism $\phi \in \GA_3(\IC[x]) \cap \tilde{\psi} ^*(\GA_3(\IC[x]))$ with $\phi(y)=b_1$.
\end{theorem}
This follows from the following technical result, whose proof we momentarily defer:
\begin{theorem}\label{psistar}
Let $\alpha \in \GA_3(\IC[x])$.  
\begin{enumerate}
\item If $\tilde{\psi}^{-1} \circ \alpha \circ \tilde{\psi} \in \GA_3(\IC[x])$, then $\alpha(p) \in (p,x^2)$.
\item Suppose that $\alpha(y)\equiv y \pmod{x}$.  Then $\tilde{\psi}^{-1} \circ \alpha \circ \tilde{\psi} \in \GA_3(\IC[x])$ if and only if $\alpha(p)=p+x^2F+xpA+x^3B$ for some $A,B \in \IC[x]^{[3]}$ and $F \in \IC[y,p]$ such that $\alpha \equiv \exp(FD) \pmod{x}$ (recall $D=y\delz-2z\delu$ from (\ref{D})).
\item Suppose that $\alpha(y) \equiv y \pmod{x}$, $\phi = \tilde{\psi}^{-1} \circ \alpha \circ \tilde{\psi} \in \GA_3(\IC[x])$, and $\phi(p) \equiv p \pmod{x}$.  If $\alpha$ is stably tame over $\IC[x]$, then $\phi$ is stably tame over $\IC[x]$ as well.
\end{enumerate}
\end{theorem}
\begin{proof}[Proof of Theorem \ref{b1fail}]
Suppose $\phi \in \GA_3(\IC[x]) \cap \tilde{\psi} ^* (\GA_3(\IC[x]))$ and $\phi(y)=b_1$.  Then $\phi = \tilde{\psi} ^{-1} \circ \alpha \circ \tilde{\psi}$ for some $\alpha \in \GA_3(\IC[x])$.  In particular, since $\phi(y)=b_1$, we may write
\begin{equation}\label{alpha}
\alpha=\left(y+xz,\sum _{i=0} x^i G_i, \sum _{i=0} x^i H_i\right)
\end{equation}
for some $G_i, H_i \in \IC^{[3]}$.  By Theorem \ref{psistar} (2), we have $\alpha \equiv \exp(FD) \pmod{x}$ for some $F \in \IC[y,p]$.  Thus
\begin{align}\label{G0H0}
G_0 &= z+yF & H_0 &= u-2zF-yF^2
\end{align}
Write $\alpha(p)=\sum _{i=0} x^i P_i$ for some $P_i \in \IC^{[3]}$.  On the one hand, we can compute directly from (\ref{alpha})
\begin{align}
P_0&=yH_0+G_0^2 \notag\\
P_1 &= yH_1+zH_0+2G_0G_1 \label{p1}\\
P_2 &= yH_2+zH_1+2G_0G_2+G_1^2 \label{p2}
\end{align}
On the other hand, from Theorem \ref{psistar} (2), we must have $P_0=p$, $P_1 \in (p)$, and $P_2 \equiv F \pmod{p}$.  We will use these to derive a contradiction.  
\begin{claim} \label{G1}
$G_1 \equiv -\frac{1}{2}u \pmod{y,z}$.
\end{claim}
\begin{proof}
Since $P_1 \equiv 0 \pmod{p}$, we apply (\ref{p1}) and compute
\begin{align*}
0 & \equiv yH_1+zH_0+2G_0G_1 & \pmod{p} \\
&\equiv yH_1+z(u-2zF-yF^2)+2(z-yF)G_1 &\pmod{p} 
\end{align*}
with the second line following from (\ref{G0H0}).  Noting that $p=yu+z^2 \in (y,z^2)$, we may go modulo $(y,z^2)$ and obtain $0 \equiv z(u+2G_1) \pmod{y,z^2}$, hence $G_1 \equiv -\frac{1}{2}u \pmod{y,z}$.
\end{proof}
Now, since $P_2 \equiv F \pmod{p}$, we apply (\ref{p2}) and compute
\begin{align*}
F &\equiv yH_2+zH_1+2(z+yF)G_2+G_1^2 & \pmod{p} 
\end{align*}
Similar to above, since $p \in (y,z)$, we see $ F \equiv G_1 ^2 \pmod{y,z}$.  Applying Claim \ref{G1}, we then have 
\begin{align*}
F \equiv \frac{1}{4}u^2 & \pmod{y,z} 
\end{align*}
However, $F \in \IC[y,p]$, so $F \equiv a \pmod{y,z}$ for some $a \in \IC$, a contradiction.
\end{proof}

So while $b_1$ cannot be a coordinate of an automorphism of $\GA_3(\IC[x]) \cap \tilde{\psi}^*(\GA_3(\IC[x]))$, we pose the following:
\begin{conjecture}$b_1$ is a coordinate if and only if it is a coordinate of an automorphism $\phi \in \GA_3(\IC[x]) \cap \psi^*(\GA_3(\IC[x]))$.
\end{conjecture}

We devote the rest of this section to proving Theorem \ref{psistar}.  First, we require a couple lemmas.
\begin{lemma}\label{ideals}
$\IC[y,v,w] \cap (x)\IC[x][y,z,u] = \IC[y,v,w] \cap (x^2)\IC[x][y,z,u] = (yw+v^2)\IC[y,v,w]$
\end{lemma}
\begin{proof}
Clearly the relation $yw+v^2=x^2p$ (from \eqref{x2p}) guarantees $$\IC[y,v,w] \cap x\IC[x][y,z,u] \supset \IC[y,v,w] \cap x^2\IC[x][y,z,u] \supset (yw+v^2) \IC[y,v,w]$$  So we simply need to see $\IC[y,v,w] \cap x\IC[x][y,z,u] \subset (yw+v^2) \IC[y,v,w]$.  Note that we have a map $\alpha: \IC[y,v,w] \rightarrow \IC[y,yp,-yp^2]$ obtained from going$\mod x$.  Clearly $yw+v^2$ is in the kernel of this map, so it descends to the quotient.  Observing that $\IC[y,v,w]/(yw+v^2) \cong \IC[y,v,\frac{-v^2}{y}]$, we have the following commutative diagram:

\begin{diagram}
\IC[y,v,w] & \rTo^\alpha & \IC[y,yp,-yp^2] \\
\dTo^\beta & \ruTo^\gamma & \\
\IC[y,v,\frac{-v^2}{y}] & & \\
\end{diagram}

Note that $\gamma$ is in fact an isomorphism; hence $\ker \beta = \ker \alpha$.  But $\beta$ is the quotient map, so $\ker \beta = (yw+v^2)$, and $\ker \alpha = \IC[y,v,w] \cap (x)\IC[x][y,z,u]$.

\end{proof}
\begin{corollary}\label{idealscorollary}
$\tilde{\psi}^{-1}(x^2\IC[x][y,z,u]) \cap \IC[x][y,z,u] = (x^2,p)\IC[x][y,z,u]$
\end{corollary}
\begin{proof}
Applying $\tilde{\psi}$ throughout, this is equivalent to showing 
\begin{equation*}
x^2\IC[x][y,z,u] \cap \IC[x][y,v,w]=(x^2,yw+v^2)\IC[x][y,v,w]
\end{equation*}

That the right side is contained in the left is immediate (recall from (\ref{x2p}) that $yw+v^2 \in x^2\IC[x][y,z,u]$).  For the opposite containment, suppose $A \in x^2\IC[x][y,z,u] \cap \IC[x][y,v,w]$ and write $A=\sum _{i=0} x^iA_i$ for some $A_i \in \IC[y,v,w]$.  It suffices to assume that $A=A_0+xA_1$.  Since $A \in x^2\IC[x][y,z,u] \subset x\IC[x][y,z,u]$, and (trivially) $xA_1 \in x\IC[x][y,z,u]$, we must also have $A_0 \in x\IC[x][y,z,u]$.  So $A_0 \in x\IC[x][y,z,u] \cap \IC[y,v,w] = (yw+v^2)\IC[y,v,w]$ by Lemma \ref{ideals}.  Thus, we may now assume $A=xA_1$.  Since $A \in x^2\IC[x][y,z,u]$, we see $A_1 \in x\IC[x][y,z,u]$, and again applying Lemma \ref{ideals} yields $A_1 \in (yw+v^2)\IC[y,v,w]$.
\end{proof}

\begin{lemma}\label{nagatatype}
Suppose $\phi \in \GA_3(\IC)$ with $\phi(y)=y$ and $\phi(p)\in (p)$.  Then $\phi = (y,cz,c^2u) \circ \exp(PD)$ for some $c \in \IC^*$, $P \in \IC[y,p]$.
\end{lemma}
\begin{proof}
Write $\phi = (y,F_1,F_2)$ for some $F_1, F_2 \in \IC^{[3]}$.  Since $\phi(p) \in (p)$ and $p$ is irreducible, we see $yF_2+F_1=rp$ for some $r \in \IC^*$.  Set $c=\frac{1}{r}J\phi \in \IC^*$.  Now we compute
$$cy=\frac{1}{r}yJ(y,F_1,F_2)=\frac{1}{r}J(y,F_1,yF_2+F_1^2)=J(y,F_1,p)$$
Observe that $J(y,\cdot,p)=D$ (recall from (\ref{D}) $D=y\delz-2z\delu$), so we may rewrite this as $D(F_1)=cy \in \ker D$; thus $F_1 = cz+\tilde{P}$ for some $\tilde{P} \in \ker D = \IC[y,p]$.  We now recompute 
\begin{equation*}
rp=\phi(p)=yF_2+(cz+\tilde{P})^2
\end{equation*}
Comparing the $z^2$ terms on each side, we deduce $r=c^2$.  We also must have $y|\tilde{P}$.  Set $P=\frac{\tilde{P}}{cy} \in \IC[y,p]$.  Plugging this back in to the relation $\phi(p)=rp=c^2p$, we obtain 
\begin{equation*}
c^2yu=yF_2+2c^2yzP+c^2y^2P^2
\end{equation*}
Thus $F_2=c^2(u-2zP-yP^2)$, and we have $\phi=\left(y,c(z+yP),c^2(u-2zP-yP^2)\right)=(y,cz,c^2u) \circ \exp(PD)$.
\end{proof}

We now have the required tools, and may proceed with the proof of Theorem \ref{psistar}.

\begin{proof}[Proof of Theorem \ref{psistar}]
Define $\phi=\tilde{\psi}^{-1} \circ \alpha \circ \tilde{\psi}$.  For (1), suppose $\phi \in \GA_3(\IC[x])$; in particular, $\phi(p) \in \IC[x]^{[3]}$.  We thus compute, noting that \eqref{psitilde} and \eqref{x2p} imply $\tilde{\psi}^{-1}(p)=\frac{p}{x^2}$,
\begin{align*}
(\tilde{\psi} ^{-1} \circ \alpha \circ \tilde{\psi})(p) &\in \IC[x][y,z,u] \\
(\alpha \circ \tilde{\psi})(\frac{p}{x^2}) &\in \IC[x][y,z,u] \\
\alpha(\frac{p}{x^2}) &\in \tilde{\psi} ^{-1}(\IC[x][y,z,u]) \\
\alpha(p) &\in \tilde{\psi} ^{-1}(x^2\IC[x][y,z,u]) 
\end{align*}
Since $\alpha \in \GA_3(\IC[x])$, we thus have $\alpha(p) \in \tilde{\psi}^{-1}(x^2\IC[x][y,z,u]) \cap \IC[x][y,z,u]$.  Applying Corollary \ref{idealscorollary}, we thus have $\alpha(p) \in (x^2,p)\IC[x][y,z,u]$, establishing assertion (1) of the theorem.

We now assume for the remainder that $\alpha(y)\equiv y \pmod{x}$.  Write 
\begin{align*}
\alpha &= (y+xQ,\sum _{i=0} x^iG_i, \sum _{i=0}x^i H_i)
\end{align*}for some $Q \in \IC[x]^{[3]}$, $G_i,H_i \in \IC^{[3]}$.  Let 
\begin{align}\label{QP}
Q&=\sum _{i=0} x^iQ_i & \alpha(p)&=\sum _{i=0} x^i P_i
\end{align}
for some $Q_i, P_i \in \IC^{[3]}$.  Direct computation shows
\begin{align}
P_0 &= yH_0+G_0^2  &
P_1 &= yH_1+Q_0H_0+2G_0G_1 \label{p01}
\end{align}

For one direction of (2), assume $\phi \in \GA_3(\IC[x])$.  In particular, $\phi(p) \in \IC[x]^{[3]}$.  We then compute, using (\ref{QP}), 
\begin{align*}
\phi(x^2p)=(\tilde{\psi}^{-1} \circ \alpha \circ \tilde{\psi})(x^2p) = (\alpha \circ \tilde{\psi})(p)=\tilde{\psi}\left(\sum _{i=0} x^{i}P_i\right) 
\end{align*}
Since $\phi(x^2p) \in x^2\IC[x]^{[3]}$, we thus see $P_0+xP_1 \in \tilde{\psi}^{-1}(x^2\IC[x]^{[3]}) \cap \IC[x]^{[3]} = (x^2,p)\IC[x]^{[3]}$ by Corollary \ref{idealscorollary}.  Since $P_0, P_1 \in \IC^{[3]}$, we then see $P_0+xP_1 \in (p)\IC[x]^{[3]}$, whence $P_0, P_1 \in (p)\IC^{[3]}$.

Consider $\bar{\alpha}=(y,G_0,H_0) \in \GA_3(\IC)$, the image modulo $x$; Since $p$ is irreducible and $\bar{\alpha}(p) = P_0 \in (p)$, by Lemma \ref{nagatatype} $\bar{\alpha}=(y,cz,c^2u) \circ \exp(FD)$ for some $c \in \IC^*$, $F \in \IC[y,p]$.   Thus we have
\begin{align*}
G_0 &= c(z+yF) & 
H_0 &= c^2(u-2zF-yF^2) 
\end{align*}
In particular, we see from (\ref{p01}) that we must have $P_0=c^2p$.  Since $\alpha(p) \in (p,x^2)$ by part (1) and $P_1 \in (p)C^{[3]}$, we can write $\alpha(p)=c^2p+xpA^\prime+ x^2B^\prime$ for some $A^\prime, B^\prime \in \IC[x]^{[3]}$.  Now, 
\begin{align*}
\phi(xz) &= \phi(v-yp) \\
&= (\tilde{\psi} ^{-1} \circ \alpha \circ \tilde{\psi}) (v-yp) \\
&= (\alpha \circ \tilde{\psi})(z-y\frac{p}{x^2}) \\
&=\tilde{\psi}\left(\sum _{i=0} x^i G_i -(y+xQ)(\frac{p}{x^2}(c^2+xA^\prime)+B^\prime)\right) \\
&=\sum _{i=0} x^i \tilde{\psi}(G_i) -(y+x\tilde{\psi}(Q))\left(p(c^2+x\tilde{\psi}(A^\prime))+\tilde{\psi}(B^\prime)\right) 
\end{align*}
with the first equality arising from (\ref{pvw}).  Since $\phi(z) \in \IC[x]^{[3]}$, we must have $\phi(xz) \in x\IC[x]^{[3]}$, and thus 
\begin{align*}
0&\equiv \tilde{\psi}(G_0)-y(pc^2)-y\tilde{\psi}(B^\prime)& \pmod{x} \\
0&\equiv \tilde{\psi}(c(z+yF))-y(pc^2)-y\tilde{\psi}(B^\prime) & \pmod{x} \\
0&\equiv c(v-ypc)+y\tilde{\psi}(cF-B^\prime) & \pmod{x} \\
0&\equiv cyp(1-c)+y\tilde{\psi}(cF-B^\prime) & \pmod{x}
\end{align*}

Observing that $\tilde{\psi}((y,z,u)\IC[x]^{[3]}) \subset (x,y)\IC[x]^{[3]}$ (coming from the fact that $y,v,w \in (x,y)\IC[x]^{[3]}$), we must have $c=1$.  We also must have $F-B^\prime \in \IC[x]^{[3]} \cap \tilde{\psi}^{-1}(x\IC[x]^{[3]})=(x,p)\IC[x]^{[3]}$ (from Corollary \ref{idealscorollary}).  Thus we write $B^\prime=F+xB+pA^{\prime \prime}$ for some $B \in \IC[x]^{[3]}$.  Setting $A=A^\prime+xA^{\prime \prime}$ gives $\alpha(p)=p+x^2F+xpA+x^3B$ as required.

For the converse, assume $\alpha(p)=p+x^2F+xpA+x^3B$ and $\bar{\alpha} \equiv \exp(FD) \pmod{x}$.  In particular, we have
\begin{align}\label{G0H02}
G_0 &= z+yF & H_0 &= u-2zF-yF^2
\end{align}
Since $J\phi = J\alpha = J\bar{\alpha}=1$, it suffices to check that $\phi \in \MA_3(\IC[x])$; compute
\begin{align*}
\phi(y)&=(\tilde{\psi} ^{-1} \circ \alpha \circ \tilde{\psi})(y)=(\alpha \circ \tilde{\psi})(y) =y+xQ(y,v,w) \in \IC[x]^{[3]}
\end{align*}
Next, we show $\phi(xz) \in x\IC[x]^{[3]}$.  Again, using \eqref{pvw},
\begin{align*}
\phi(xz) &= (\tilde{\psi} ^{-1} \circ \alpha \circ \tilde{\psi})(v-yp)\\
&=(\alpha \circ \tilde{\psi})\left(z-y\frac{p}{x^2}\right) \\
&=\tilde{\psi}\left(\sum _{i=0}x^i G_i-\left(y+xQ\right)\left(\frac{1}{x^2}\right)\left(p+x^2F+xpA+x^3B\right)\right) \\
&=\sum _{i=0}x^i \tilde{\psi}(G_i)-\left(y+x\tilde{\psi}(Q)\right)\left(p+\tilde{\psi}(F)+x(p\tilde{\psi}(A)+x\tilde{\psi}(B))\right)
\end{align*}
So $\phi(xz) \in \IC[x]^{[3]}$.  Going modulo $x$, and recalling from (\ref{G0H02}) that $G_0=z+yF$ (and thus $\tilde{\psi}(G_0)=v+y\tilde{\psi}(F)$), we obtain
\begin{align*}
\phi(xz) &\equiv (v+y\tilde{\psi}(F))-y(p+\tilde{\psi}(F))\equiv 0 & \pmod{x} 
\end{align*}
So we have $\phi(xz) \in x\IC[x]^{[3]}$ and thus $\phi(z) \in \IC[x]^{[3]}$.  Since $\phi(y) \notin x\IC[x]^{[3]}$, it suffices to check that $\phi(p) \in \IC[x]^{[3]}$ as well (as then $\phi(u) \in \IC[x]^{[3]}$).   
\begin{align*}
\phi(p) &= (\tilde{\psi}^{-1} \circ  \alpha \circ \tilde{\psi}) (p) \\
&= (\alpha \circ \tilde{\psi}) (\frac{p}{x^2}) \\
&=\tilde{\psi}\left(\frac{1}{x^2}(p+x^2F+xpA+x^3B)\right) \\
&=p+\tilde{\psi}(F)+x(p\tilde{\psi}(A)+\tilde{\psi}(B))
\end{align*}
So $\phi(p) \in \IC[x]^{[3]}$, and thus $\phi \in \GA_3(\IC[x,x^{-1}]) \cap \MA_3(\IC[x])$.  Since $J\phi=J\alpha = 1$, by Theorem \ref{overring} $\phi \in \GA_3(\IC[x])$.

For the final part, we will use Theorem \ref{BEWtheorem} and the following well known result of Smith \cite{Smith}:
\begin{theorem}[Smith]
Let $A$ be a $\IQ$-algebra and $D$ be a triangular derivation of $A^{[n]}$.  Then for any $P \in \ker D$, $\exp(PD)$ is 1-stably tame.
\end{theorem}

Suppose $\phi$, $\alpha$ are as in (3).  While $\phi \notin \TA_3(\IC[x,x^{-1}])$, Smith's theorem shows that $\tilde{\psi}$ is 1-stably tame (over $\IC[x,x^{-1}]$), hence $\phi$ is stably tame over $\IC[x,x^{-1}]$.  One also quickly checks that $J\phi=J\alpha\in \IC^*$.  So we only need to see that $\bar{\phi}$ is stably a composition of elementaries.  By Lemma \ref{nagatatype}, since $\bar{\phi}(y)=y$ and $\bar{\phi}(p)=p$, a composition with a diagonal map allows us to assume $\bar{\phi}(p) = \exp(PD)$ for some $P \in \IC[y,p]$ (and thus $J\phi = J\bar{\phi}=1$).  Again appealing to Smith's theorem, we thus have $\bar{\phi}(p) \in \EA_4(\IC)$.  Thus by Theorem \ref{BEWtheorem}, $\phi$ is stably tame.
\end{proof}

\section{\Vtype Polynomials}
Instead of only considering the \Venereau polynomials, one may generalize these slightly and still retain all the coordinate-like properties. 

\begin{definition}
A {\em \Vtype polynomial} is a polynomial of the form $y+xQ$ for some $Q \in \IC[x][v,w]$.
\end{definition}

Note that $Q=x^{m-1}v$ gives the \Venereau polynomials $b_m$.  We first give a sufficient condition for a \Vtype polynomial to be a coordinate.

\begin{theorem}\label{coordinate}
If $Q = x^2Q_1 + xvQ_2$ for some $Q_1 \in \IC[x][v,w]$ and $Q_2 \in \IC[v^2,w]$, then the \Vtype polynomial $y+xQ$ is a coordinate.
\end{theorem}

\begin{proof}
Write $Q_2 = \sum \alpha _{a,b} v^{2a}w^b$ for some $\alpha _{a,b} \in \IC$.  Define  $\alpha \in \GA_3(\IC[x])$ by
\begin{align*}
\alpha &= \left(y+xQ(z,u),z-\frac{1}{2}x^2\sum \alpha _{a,b} (-1)^ay^au^{a+b+1},u\right)
\end{align*}
Direct computation shows
\begin{align*}
\alpha(p) &\equiv \left(y+xQ(z,u)\right)u+\left(z-\frac{1}{2}x^2 \sum \alpha _{a,b} (-1)^ay^au^{a+b+1}\right)^2  &\pmod{x^3} \\
&\equiv (yu+z^2)+x^2z\left(uQ_2(z,u)-\sum \alpha _{a,b} (-1)^a y^a u^{a+b+1}\right) & \pmod{x^3} \\
&\equiv p+x^2z\left(\sum \alpha _{a,b}u^{b+1} (z^{2a} -(-yu)^a)\right) & \pmod{x^3} 
\end{align*}
Noting that $z^{2a}-(-yu)^a \in (z^2+yu) = (p)$, we thus have that $\alpha(p) \equiv p \pmod{(xp,x^3)}$.
Thus, by Theorem \ref{psistar} (2), $\phi := \tilde{\psi} ^{-1} \circ \alpha \circ \tilde{\psi} \in \GA_3(\IC[x])$, and one quickly checks $\phi(y)=y+xQ$.
\end{proof}
\begin{corollary}The second \Venereau polynomial $b_2$ is a coordinate.
\end{corollary}

It is interesting to note that the above automorphisms are stably tame; moreover, any automorphism of the above type, with a \Vtype polynomial as a coordinate, must be stably tame.  

\begin{theorem}
If $Q=x^2Q_1+xvQ_2$ for some $Q_1 \in \IC[x][v,w]$ and $Q_2 \in \IC[v^2,w]$, any $\varphi \in \GA_3(\IC[x])$ with $\varphi(y)=y+xQ$ is stably tame.
\end{theorem}

\begin{proof}
First, note that the $\phi$ constructed in the proof of Theorem \ref{coordinate} is stably tame; indeed, the fact that $\alpha(p) \equiv p \pmod{(xp,x^3)}$ guarantees $\phi(p) \equiv p \pmod{x}$, thus $\phi$ is stably tame by part 3 of Theorem \ref{psistar}.
Now, consider arbitrary $\varphi \in \GA_3(\IC[x])$ with $\varphi(y)=y+xQ$.  Then $(\varphi \circ \phi ^{-1})(y) = (\phi \circ \phi ^{-1}) (y)=y$.  In other words, $\varphi \circ \phi ^{-1} \in \GA_2(\IC[x][y])$.  
The main result of \cite{BEW} states that any automorphism in two variables over a regular ring is stably tame; since $\IC[x]$ is regular, so is $\IC[x][y]$, thus the composition $\varphi \circ \phi ^{-1}$ is stably tame.  Since $\phi$ is also stably tame, we thus see that $\varphi$ is stably tame.
\end{proof}

Among the \Vtype polynomials, the {\em \Venereau complements} are worth highlighting:
\begin{definition}The polynomials $c_m := y+x^mw$ ($m \geq 1)$ are called {\em \Venereau complements}
\end{definition}

\begin{theorem}\label{bc}
The \Venereau polynomial $b_m$ is a coordinate if and only if the \Venereau complement $c_{2m}$ is.
\end{theorem}
\begin{proof}
Define $\theta _m \in \GA_3(\IC[x,x^{-1}])$ by
\begin{align*}
\theta _m &= \tilde{\psi}^{-1} \circ (y+x^mz-\frac{1}{4}x^{2m}u,z-\frac{1}{2}x^{m}u,u) \circ \tilde{\psi}
\end{align*}
One easily checks that $\theta _m (y)=b_m-\frac{1}{4}x^{2m}w$, $\theta _m (w)=w$, and via Theorem \ref{psistar}, $\theta _m \in \GA_3(\IC[x])$ for each $m \geq 1$.  It thus follows that $b_m$ is a coordinate if and only if $y+\frac{1}{4}x^{2m}w$ is, as $(y+\frac{1}{4}x^{2m}w,*,*) \circ \theta _m = (b_m,*,*)$.  Consider the automorphism $\beta_m \in \GA_4(\IC)$ given by $\beta _m=(\lambda x,y,\lambda z, \lambda ^2u)$ where $\lambda ^{2m+4}=4$.  One easily checks that $\beta _m (w)=\lambda ^4 w$ and thus $\beta _m (y+\frac{1}{4}x^{2m}w)=c_{2m}$.  Thus conjugation by $\beta _m$ shows that $c_{2m}$ is a coordinate if and only if $y+\frac{1}{4}x^{2m}w$ is as well.
\end{proof}

Note that by Theorem \ref{coordinate}, $c_m$ is a coordinate for $m \geq 3$.  So this slightly stronger fact gives an alternate proof of the fact that $b_m$ is a coordinate for all $m \geq 2$. 

\begin{remark}
Theorem \ref{bc} can be generalized slightly: let $P(w) \in \IC[x][w]$.  Then $y+xvP(w)$ is a coordinate if and only if $y+x^2wP(w)^2$ is.  The proof is almost identical.
\end{remark}

The remainder is devoted to showing that all \Vtype polynomials satisfy a variety of coordinate-like properties that are known to hold for the \Venereau polynomial.  In particular, we show that they are strongly $x$-residual coordinates (Proposition \ref{VTypeResidual}), hyperplanes (Theorem \ref{hyperplane}), hyperplane fibrations (Theorem \ref{hyperplanefibration}), and stably tame 1-stable coordinates (Theorem \ref{stablytamestable}).

\begin{proposition}\label{VTypeResidual}
Let $f=y+xQ$, $Q \in \IC[x][v,w]$ be a \Vtype polynomial.  Then $f$ is a strongly $x$-residual coordinate.
\end{proposition} 

\begin{proof}
Since $Q \in \IC[x][v,w]$, $\phi=\tilde{\psi}^{-1} \circ (y+xQ(z,u),z,u) \circ \tilde{\psi} \in \GA_3(\IC[x,x^{-1}])$ and $\phi(y)=y+xQ$, so $y+xQ$ is an $\IC[x,x^{-1}]$-coordinate.  Clearly $\bar{f} \equiv y \pmod{x}$, so $f$ is a strongly $x$-residual coordinate.
\end{proof}

To see that \Vtype polynomials are hyperplanes, we use the following fact (pointed out to me by Arno van den Essen), which also appears (with $b=0$) in \cite{KVZ}.  A special case of this was used in \cite{Vthesis} to show that $f_1$ is a hyperplane.

\begin{lemma} \label{arnolemma}
Let $A$ be a commutative ring, $a,b \in A$, and $g \in A[y]$.  Then $A[y]/(y+ag(y)-b) \cong A[y]/(y+g(ay+b))$.
\end{lemma}

\begin{proof}
We compute below, where the first isomorphism is given by sending $y$ to $ay+b$; we then identify $t=-g(ay+b)$, and use this to rewrite the relation in terms of $t$ only.
\begin{align*}
A[y]/(y+ag(y)-b) &\cong A[ay+b]/\left(a(y+g(ay+b))\right) \\
&\cong A[ay+b,t]/\left(a(y+g(ay+b)),t+g(ay+b)\right) \\
&\cong A[ay+b,t]/\left(a(y-t),t+g(ay+b)\right) \\
&\cong A[t]/\left(t+g(at+b)\right)
\end{align*}
\end{proof}

\begin{corollary} \label{arno}
Let $A$ be a commutative ring, $a,b \in A^{[n]}$ and $g \in A^{[n]}[y]$.  Then $y+ag(y)-b$ is a hyperplane (over $A$) if and only if $y+g(ay+b)$ is as well.
\end{corollary}

\begin{theorem}\label{hyperplane}
Let $f=y+xQ$, $Q \in \IC[x][v,w]$, be a \Vtype polynomial.  Then $f$ is an $\IC[x]$-hyperplane of $\IC[x]^{[3]}$; that is, $\IC[x]^{[3]}/(f) \cong _{\IC[x]} \IC[x]^{[2]}$.
\end{theorem}

\begin{proof}
Similar to (\ref{pvw}), define
\begin{align}\label{pvw0}
p_0&=xyu+z^2 & v_0&=z+yp_0 & w_0& =xu-2v_0p_0+yp_0^2
\end{align}
\begin{claim}It suffices to check that $y+xQ_0$ is a coordinate for any $Q_0 \in \IC[x][v_0,w_0]$.
\end{claim}
\begin{proof}
Applying Corollary \ref{arno} to $f$ (with $a=x$ and $b=0$) yields that $f$ is a hyperplane if and only if $f_0=y+Q(xv_0,xw_0)$ is.  We can write $Q(xv_0,xw_0)=xQ_0+\lambda$ for some $Q_0 \in \IC[x][v_0,w_0]$ and $\lambda \in \IC$; hence $f_0=y+xQ_0+\lambda$.  Thus, if $y+xQ_0$ is a coordinate, so is $f_0$, hence $f_0$ is a hyperplane and (by Corollary \ref{arno}) so is $f$.
\end{proof}
We have thus reduced the theorem to showing that $y+xQ_0$ is a coordinate for any $Q_0 \in \IC[x][v_0,w_0]$.  The proof of this is quite analagous to that of Theorem \ref{f2}.  First, define $D_0 = xy \delz-2z\delu$, and observe that $\ker D_0 = \IC[x][y,p_0]$.  Also define
\begin{equation*}
\psi _0 = \exp\left(\frac{p_0}{x}D_0\right) = \left(y,v_0, \frac{w_0}{x}\right)
\end{equation*}
In particular, since $\psi _0$ fixes $p_0$, we see
\begin{equation}\label{p0}
yw_0+v_0^2=p_0
\end{equation}
Also define 
\begin{align*}
\alpha _0 &:= (y+xQ_0(z,u),z,u) & \phi _0 &:= \psi _0 ^{-1} \circ \alpha _0 \circ \psi _0
\end{align*}
One immediately sees that $\phi _0 \in \GA_3(\IC[x,x^{-1}])$, $J\phi _0 =1$, and $\phi _0 (y)=y+xQ_0$.  So by Theorem \ref{overring}, we need only see that $\phi _0 \in \MA_3(\IC[x])$.  We first compute $\phi _0 (v_0)$, $\phi _0 (w_0)$, and $\phi _0 (p_0)$, and use those to compute $\phi_0(z)$, $\phi_0(u)$.
\begin{align*}
\phi _0 (v_0) &= (\psi _0 ^{-1} \circ \alpha _0 \circ \psi _0) (v_0) &
\phi _0 (w_0) &= (\psi _0 ^{-1} \circ \alpha _0 \circ \psi _0) (w_0) \\
&= (\alpha _0 \circ \psi _0) (z) &
&= (\alpha _0 \circ \psi _0) (xu) \\
&= (\psi _0) (z) &
&= (\psi _0) (xu) \\
&=v_0 & &=w_0
\end{align*}
Now from (\ref{p0}) we see
\begin{equation*}
\phi_0 (p_0) = \phi_0(yw_0+v_0^2)
= (y+xQ_0)w_0+v_0^2 
=p_0+xQ_0w_0
\end{equation*}
Finally, we compute, using \eqref{pvw0}
\begin{align*}
\phi_0 (z) &= \phi_0(v_0 -yp_0) \\
&= v_0-(y+xQ_0)(p_0+xQ_0w_0) 
\end{align*}
and
\begin{align*}
\phi _0 (u) &= \frac{1}{x} \phi_0(w_0+2v_0p_0-yp_0^2) \\
&= \frac{1}{x}\left(w_0+2v_0(p_0+xQ_0w_0)-(y+xQ_0)(p_0+xQ_0w_0)^2\right) \\
&= \frac{1}{x}\left(w_0+2v_0p_0-yp_0^2\right)+Q_0(2w_0v_0-2w_0yp_0-p_0^2)+ \\
&\phantom{xxx}-xQ_0^2(yw_0^2-2w_0p_0)-x^2Q_0^3w_0^2 \\
&= u+Q_0(2w_0v_0-2w_0yp_0-p_0^2)-xQ_0^2(yw_0^2-2w_0p_0)-x^2Q_0^3w_0^2 
\end{align*}
Thus $\phi _0 \in \MA_3(\IC[x])$ and hence $\phi _0 \in \GA_3(\IC[x])$.
\end{proof}

We even have something stronger, namely that all \Vtype polynomials define hyperplane fibrations.  To show this for $b_1$, \Venereau made use of the fact that $b_1=y+x(xz+y(yu+z^2))$ is linear in $u$; this does not hold for \Vtype polynomials in general, so we have to do a bit more work.
\begin{theorem}\label{hyperplanefibration}
Let $f=y+xQ$, $Q \in \IC[x][v,w]$, be a \Vtype polynomial.  Then for each $c \in \IC$, $f-c$ is an $\IC[x]$-hyperplane of $\IC[x]^{[3]}$.
\end{theorem}
\begin{proof}
By Theorem \ref{hyperplane}, we assume $c \neq 0$.  As in the proof of that theorem, we would like to apply Corollary \ref{arno}.  However, now we must do so with $b=c$ (still $a=x$).  Define
\begin{align}\label{pvw1}
p_1&=(xy+c)u+z^2 & v_1&=xz+(xy+c)p_1 & w_1& =x^2u-2v_1p_1+(xy+c)p_1^2
\end{align}
Then Corollary \ref{arno} yields that we need only show $y+Q(v_1,w_1)$ is a hyperplane (in fact, we show it is a coordinate).  The general idea is as follows: construct an automorphism $(y+Q(v_1,w_1),*,*) \in \GA_3(\IC[x,x^{-1}])$, then compose on the left with automorphisms fixing $y$ until it is also in $\MA_3(\IC[x,x^{-1}])$.  Then we apply Theorem \ref{overring}.  We start by defining a derivation over $\IC[x,x^{-1}]$ by
\begin{equation*}
D_1 = (xy+c)\delz-2z\delu
\end{equation*}
$D_1$ is triangular, and hence locally nilpotent.  One easily checks that $\ker D_1 = \IC[x][y,p_1]$.  Define
\begin{equation*}
\phi _0 = \exp(\frac{p_1}{x}D_1) = (y,\frac{v_1}{x},\frac{w_1}{x^2})
\end{equation*}
Since $x^2p_1 \in \ker D_1$, we must have $\phi _0(x^2p_1)=x^2p_1$ and thus obtain
\begin{equation}\label{x2p1}
(xy+c)w_1+v_1^2=x^2p_1
\end{equation}
Our first step is to define
\begin{equation*}
\phi _1 = (y+Q(xz,x^2u),z,u) \circ \phi _0 = (y+Q(v_1,w_1),\frac{v_1}{x},\frac{w_1}{x^2})
\end{equation*}
Next, we set
\begin{align}
\phi _2 &= (y,z,u+c^{-1}z^2) \circ \phi _1 \notag \\
&= \left(y+Q(v_1,w_1),\frac{v_1}{x}, \frac{cw_1+v_1^2}{cx^2}\right) \notag\\
&= \left(y+Q(v_1,w_1),\frac{v_1}{x}, c^{-1}p_1-\frac{yw_1}{cx}\right) \label{phi2}
\end{align}
with the last equality following from \eqref{x2p1}.  For the next step, we require the following:
\begin{claim}\label{claimG}
For any $G \in \IC[x]^{[2]}$, $G(v_1,-c^{-1}v_1^2) \equiv G(v_1,w_1)+xyG^\prime(v_1,w_1) \pmod{x^2}$ for some $G^\prime \in \IC[x]^{[2]}$.
\end{claim}
\begin{proof}
This is a straightforward computation, appealing to \eqref{x2p1}.  Indeed, note that 
\begin{equation*}
-c^{-1}v_1^2=w_1+c^{-1}x(yw_1-xp_1)
\end{equation*}
We thus compute, applying Taylor's formula at the second step,
\begin{align*}
G(v_1,-c^{-1}v_1^2) &\equiv G(v_1,w_1+c^{-1}xyw_1) & \pmod{x^2} \\
&\equiv G(v_1,w_1)+c^{-1}xyw_1\del{G}{w_1}(v_1,w_1) & \pmod{x^2}
\end{align*}
Setting $G^\prime = c^{-1}w_1\del{G}{w_1}(v_1,w_1)$ yields the claim.
\end{proof}
\begin{claim}\label{claimH}
For any $H \in \IC[x]^{[3]}, G\in \IC[x]^{[2]}$, there exists $H^\prime \in \IC[x]^{[3]}$ such that
\begin{equation}\label{eqH}
H(y+G(v_1,w_1)-G(v_1,-c^{-1}v_1^2),v_1,-c^{-1}v_1^2) \equiv H(y,v_1,w_1)+xH^\prime(y,v_1,w_1) \pmod{x^2}
\end{equation}
Moreover, if $x|H$, then $x|H^\prime$.
\end{claim}
\begin{proof}
Let 
\begin{equation*}
L:=H(y+G(v_1,w_1)-G(v_1,-c^{-1}v_1^2),v_1,-c^{-1}v_1^2)
\end{equation*}
First, apply Claim \ref{claimG} to $G$, obtaining $L=H(y+xyG^\prime(v_1,w_1)+x^2T_1,v_1,-c^{-1}v_1^2)$ for some $T_1 \in \IC[x]^{[3]}$.  Next, apply Taylor's formula in the first component, obtaining for some $H^\prime \in \IC[x][y,v_1,w_1]$ and $T_2 \in \IC[x]^{[3]}$, 
\begin{align*}
L&=H(y,v_1,-c^{-1}v_1^2)+xH^\prime(y,v_1,w_1)+x^2T_2 \\
&=H(y,v_1,w_1+c^{-1}xyw_1-c^{-1}x^2p_1)+xH^\prime(y,v_1,w_1)+x^2T_2 
\end{align*}
with the second equality following from \eqref{x2p1}.  We now apply Taylor's formula in the third component, obtaining
\begin{align*}
L&\equiv H(y,v_1,w_1)+xH^{\prime \prime} (y,v_1,w_1) & \pmod{x^2}
\end{align*}
This is precisely the desired claim.
\end{proof}
\begin{claim}Suppose $\phi \in \GA_3(\IC[x,x^{-1}])$ is of the form
\begin{equation*}
\phi = \left(y+F(v_1,w_1),\frac{v_1}{x},A+x^{r}B+x^{r-2}G(y,v_1,w_1)\right)
\end{equation*}
for some $r \in \IZ$, $A,B \in \IC[x]^{[3]}$, $F \in \IC[x][v_1,w_1]$, and $G \in \IC[x][y,v_1,w_1]$.  Then there exists $\tau \in \GA_3(\IC[x,x^{-1}])$ such that $\tau \circ \phi = (y+F(v_1,w_1),\frac{v_1}{x},A+x^rB^\prime)$ for some $B^\prime \in \IC[x]^{[3]}$.
\end{claim}
\begin{proof}
Define 
\begin{equation*}
\tau = \left(y,z,u-x^{r-2}G\left(y-F(xz,-c^{-1}(xz)^2), xz,-c^{-1}(xz)^2\right)\right)
\end{equation*}
Now compute
\begin{align*}
\tau \circ \phi = \big(y&+F(v_1,w_1),\frac{v_1}{x}, A+x^rB+\\
&x^{r-2}\left(G(y,v_1,w_1)-G\left(y+F(v_1,w_1)-F(v_1,-c^{-1}v_1^2),v_1,-c^{-1}v_1^2\right)\right) \big) 
\end{align*}
Applying Claim \ref{claimH}, we obtain
\begin{equation*}
\tau \circ \phi = \left(y+F(v_1,w_1),\frac{v_1}{x}, A+x^rB^\prime+x^{r-1}G^\prime(y,v_1,w_1)\right)
\end{equation*}
Note that, by Claim \ref{claimH}, if $x|G$, then $x|G^\prime$.

Now we can apply an induction step (replacing $B$ with $B^\prime$ and $G$ with $xG^\prime$) to obtain the desired result.
\end{proof}
Now, apply the preceding claim to $\phi _2$ \eqref{phi2} to obtain
\begin{align*}
\phi _3 := \left(y+Q(v_1,w_1),\frac{v_1}{x},c^{-1}p_1+xB\right)
\end{align*}
for some $B \in \IC[x]^{[3]}$.  Finally, set 
\begin{align*}
\phi _4 &= (y,z-\frac{c^2u}{x},u) \circ \phi _3 \\
&=\left(y+Q(v_1,w_1), \frac{v_1}{x}-\frac{c}{x}(p_1+xcB),c^{-1}p_1+xB\right) \\
&=\left(y+Q(v_1,w_1), z+yp_1-c^2B,c^{-1}p_1+xB\right) 
\end{align*}
with the last equality coming from \eqref{pvw1}.  Thus we have $\phi _4 \in \MA_3(\IC[x]) \cap \GA_3(\IC[x,x^{-1}])$, hence $\phi _4 \in \GA_3(\IC[x])$ and $\phi_4(y)=y+Q(v_1,w_1)$, i.e. $y+Q(v_1,w_1)$ is a coordinate as required.
\end{proof}

Before proceeding, we remark that Corollary \ref{arno} raises an interesting question:

\begin{conjecture}
Let $A$ be a $\IQ$ algebra, $a, b \in A^{[n]}$, and $g \in A^{[n]}[y]$.  Then $y+g(ay+b)$ is a coordinate if and only if $y+ag(y)-b$ is.
\end{conjecture}

This is weaker than the Embedding Conjecture, but a proof of this conjecture would make every \Vtype polynomial a coordinate.

The next thing we prove is a generalization of a fact about $f_1$ from \cite{EMV}.  As described there, the motivation is the following well known lemma:

\begin{lemma}
Let $A$ be a ring, and $f \in A^{[n]}$.  Then the following are equivalent
\begin{enumerate}
\item $f$ is a coordinate
\item $A^{[n]} \cong _{A[f]} A[f]^{[n-1]}$
\item $A[c]^{[n]}/(f-c) \cong _{A[c]} A[c]^{[n-1]}$, where $c$ is an additional indeterminate
\end{enumerate}
\end{lemma}

In light of this, we may observe that

\begin{corollary}
Let $P(x,c) \in \IC[x,c]$, and set $\tilde{R}=\IC[x,c]/(P)$.  Let $f$ be a \Vtype polynomial.  If $f$ is a coordinate of $\IC[x][y,z,u]$, then $f - c$ is a $\tilde{R}$-hyperplane; i.e. $\tilde{R}[y,z,u]/(f-c) \cong_{\tilde{R}} \tilde{R}^{[2]}$.
\end{corollary}

So if we found one such $P$ for which a \Vtype polynomial $f$ is not a $\tilde{R}$-hyperplane, then it could not be a coordinate.

\begin{theorem}\label{cusp}
For any of the following polynomials $P_i \in \IC[x,c]$ and any \Vtype polynomial $f=y+xQ$, $f-c$ is a $\tilde{R_i}$-hyperplane (where $\tilde{R_i} := \IC[x,c]/(P_i)$).
\begin{enumerate}
\item $P_1 := x-x_0$ (where $x_0 \in \IC$)
\item $P_2 := c-c_0$ (where $c_0 \in \IC$)
\item $P_3 := x^2-c^3$
\end{enumerate}
\end{theorem}

\begin{proof}
For (1), first suppose $x_0=0$.  Then $f-c\equiv y-c \pmod x$ is obviously a coordinate, hence a hyperplane.  If $x_0 \in \IC^*$, then since $f-c$ is a $\IC[x,x^{-1}]$ coordinate, $\bar{f}-c$ is a coordinate when we go mod $x-x_0$, and hence is a hyperplane.
%
For part (2), note that $\IC[x,c]^{[3]}/(c-c_0,f-c) \cong \IC[x]^{[3]}/(f-c_0)$.  But by Theorem \ref{hyperplanefibration}, we have $\IC[x]^{[3]}/(f-c_0) \cong \IC[x]^{[2]}$ giving the desired result.

For (3), we follow the approach sketched in \cite{EMV} for the \Venereau polynomials; this requires Corollary \ref{arno} and the following (Corollary 1.31 from \cite{KVZ}), which is an application of the Abhyankar-Moh-Suzuki theorem:

\begin{lemma}\label{KVZlemma}
Let $\alpha(y,z,u) \in \IC[t][y,z,u]$ be a $\IC[t]$-coordinate, and let $\beta \in \IC[t]$.  If $\alpha(y,z,\beta u)$ is a $\IC[t]$-residual coordinate, then it is also a $\IC[t]$-coordinate, and hence a $\IC[t]$-hyperplane.
\end{lemma}

The key idea from \cite{EMV} is that $\tilde{R_3} \cong \IC[t^2,t^3]$, and using the main theorem from that paper, it suffices to show the image of $f-c$ in $\IC[t^2,t^3]$ is a $\IC[t]$-hyperplane.  Note that the image of $f-c$ is 
\begin{equation*}
f_0 = y+t^3Q(t^3z+yp, t^6u-2t^3zp-yp^2)-t^2
\end{equation*}
As in the proof of Theorem \ref{hyperplane}, we apply Corollary \ref{arno} with $a=b=t^2$; hence, it suffices to show 
\begin{equation*}
f_1 = y+t^3Q(tz+(y+1)p_1, t^4u-2tzp_1-(y+1)p_1^2)
\end{equation*}
is a coordinate over $\IC[t]$, where $p_1=t^2(y+1)u+z^2$.  In order to apply Lemma \ref{KVZlemma}, we first establish that $f_1$ is a strongly $t$-residual coordinate, and hence a residual coordinate over $\IC[t]$.  Clearly it is a coordinate modulo $t$; to see it is a coordinate over $\IC[t,t^{-1}]$, define $D_1 = t^2(y+1)\delz-2z\delu$, and set $\phi _1 = (y+t^3Q(tz,t^4u),z,u) \circ \exp(\frac{p_1}{t^3}D_1)$.  One may quickly check that $\phi_1(y)=f_1$.   

Now, taking $\beta=t^2$ in Lemma \ref{KVZlemma}, it suffices to show that $y+t^3Q(tz+(y+1)p_2,t^2u-2tzp_2-yp_2^2)$ is a coordinate, where $p_2=(y+1)u+z^2$.  But aside from replacing $t$ by $x$, we recognize this as $y+x^3Q(v,w)$ conjugated by $(y-1,z,u)$, so we are done since $y+x^3Q(v,w)$ is a coordinate by Theorem \ref{coordinate}.
\end{proof}

By Theorem \ref{Asanuma}, all \Vtype polynomials are stable coordinates.  However, no bound is given on the number of additional variables needed.  It was previously shown by Freudenburg \cite{Freud} that for $f_1$ and $f_2$, only one additional variable is needed; it turns out one is sufficient for all \Vtype polynomials.

\begin{theorem}\label{stablytamestable}
Let $f=y+xQ$, $Q \in \IC[x][v,w]$ be a \Vtype polynomial.  Then there exists $\phi \in \GA_4(\IC[x])$ with $\phi(y)=f$ and $\phi$ stably tame.
\end{theorem}

A proof of this can be found in \cite{Mythesis}, and will appear as part of a more general method in \cite{StronglyResidual}.  We remark that in the case of the \Venereau polynomial, this produces a different coordinate system than Freudenburg's.

\subsection*{Acknowledgements}
The author would like to thank his advisor, David Wright, for his numerous suggestions and helpful discussions.  The author would also like to thank Arno van den Essen for his feedback on an early manuscript, and in particular, pointing out lemma \ref{arnolemma}.

\bibliographystyle{hplain}
\bibliography{ref}

\end{document}